\documentclass[reqno]{amsart}
\usepackage{amsmath, amssymb, amsthm, epsfig}
\usepackage[hidelinks]{hyperref}
\usepackage{latexsym}
\usepackage{url}
\usepackage[mathscr]{euscript}
\usepackage{array}
\usepackage{wrapfig}
\usepackage{multirow}
\usepackage{tabularx}

\usepackage{color}
\usepackage{fullpage}
\usepackage{setspace}

\def\today{\ifcase\month\or
	January\or February\or March\or April\or May\or June\or
	July\or August\or September\or October\or November\or December\fi
	\space\number\day, \number\year}

\DeclareMathOperator{\supp}{\mathrm{supp}}
\newtheorem{problem}{Extremal Problem}
\newtheorem{theorem}{Theorem}
\newtheorem*{theoremA}{Theorem A}

\newtheorem{lemma}[theorem]{Lemma}

\theoremstyle{definition}

\theoremstyle{remark}
\newtheorem*{remark}{Remark}

\newcommand{\R}{\mathbb{R}}
\newcommand{\N}{\mathbb{N}}

\newcommand{\Z}{\mathbb{Z}}

\newcommand{\hh}{\tfrac12}

\renewcommand{\d}{\text{\rm d}}

\newcommand{\re}{{\rm Re}\,}

\title[On the $q$-analogue of the PCC]{On the $q$-analogue of the pair correlation conjecture\\ via Fourier optimization}

\author[O.E. Quesada-Herrera]{Oscar E. Quesada-Herrera}
\address{IMPA - Instituto Nacional de Matem\'{a}tica Pura e Aplicada - Estrada Dona Castorina, 110, Rio de Janeiro, RJ, Brazil 22460-320}
\email{oscarqh@impa.br}

\date{\today}
\begin{document}

\allowdisplaybreaks
	\numberwithin{equation}{section}
	\maketitle
	
	\begin{abstract}
	We study the $q$-analogue of the average of Montgomery's function $F(\alpha,\, T)$ over bounded intervals. Assuming the Generalized Riemann Hypothesis for Dirichlet $L$-functions, we obtain upper and lower bounds for this average over an interval that are quite close to the pointwise conjectured value of $1$. To compute our bounds, we extend a Fourier analysis approach by Carneiro, Chandee, Chirre, and Milinovich, and apply computational methods of non-smooth programming. 
	
	\end{abstract}
	
	\section{Introduction}
	\subsection{The Pair Correlation Conjecture}
	Let $\zeta(s)$ be the Riemann zeta-function, and assume the Riemann Hypothesis (RH). Understanding the finer aspects of the vertical distribution of the zeros $\rho=\hh+i\gamma$ of $\zeta(s)$ remains an important problem to date. While studying this distribution, Montgomery \cite{Mo}, in 1973, formulated his Pair Correlation Conjecture. It states that
	\begin{equation}\label{eq:pair_correlation}
	    \sum_{\substack{0<\gamma,\, \gamma'\le T\\
	    0<\gamma-\gamma'\le \frac{2\pi\beta}{\log T}}} 1 \sim N(T)\int_0^\beta \left\{1-\left(\frac{\sin \pi u}{\pi u}\right)^2 \right\}\, \d u, 
	\end{equation}
	as $T\to\infty,$ for any fixed $\beta >0$, where the double sum runs over the ordinates of the non-trivial zeros of $\zeta(s)$. Here, $N(T)$ denotes the number of zeros up to height $T$ (counted with multiplicity), and satisfies that, as $T\to\infty$,
\[N(T)\sim \frac{T}{2\pi}\log T.
\]
Therefore, the Pair Correlation Conjecture gives an asymptotic formula for the number of pairs of zeros whose distance is at most $\beta$ times the average gap between zeros, $\frac{2\pi}{\log T}$. \smallskip 

	Montgomery wanted to understand sums involving the differences $(\gamma-\gamma')$, such as the left-hand side of \eqref{eq:pair_correlation}. With this goal in mind, for any $R\in L^1(\R)$, let $\widehat{R}(y)=\int_{-\infty}^\infty R(x)e^{-2\pi i xy}\, \d x$ denote the Fourier transform of $R$. If $\widehat{R}\in L^1(\R)$, Fourier inversion yields the convolution formula
	\[\sum_{0<\gamma,\, \gamma'\le T}R\left(\frac{(\gamma-\gamma')\log T}{2\pi}\right)w(\gamma-\gamma') = N(T)\int_\R F(\alpha)\, \widehat{R}(\alpha)\, \d \alpha,
	\]
	where we introduce a weight $w(u):=\frac{4}{4+u^2}$, and Montgomery's function $F(\alpha)$ is the (suitably weighted and normalized) Fourier transform of the distribution function of the differences $(\gamma-\gamma')$. It is defined as
	\begin{equation*}
	    F(\alpha)= F(\alpha,\, T) := \frac{1}{N(T)}\sum_{0<\gamma, \, \gamma' \le T}T^{i\alpha(\gamma-\gamma')}w(\gamma-\gamma'), 
	\end{equation*}
	where $\alpha\in\R$ and $T\ge 15$. Thanks to the convolution formula, to understand sums over pairs of zeros, and therefore expressions such as the left-hand side of \eqref{eq:pair_correlation}, it is useful to study the asymptotic behavior of $F(\alpha)$, for large $T$. Building on Montgomery's work \cite{Mo}, Goldston and Montgomery \cite[Lemma 8]{GoMo} showed that, as $T\to\infty$,
	\begin{equation}\label{eq:F_asymp}
	    F(\alpha, \, T)=(T^{-2|\alpha|}\log T + |\alpha|)\left(1+O\left(\sqrt{\frac{\log \log T}{\log T}} \right)\right),
	\end{equation}
	uniformly for $|\alpha|\le 1$. Montgomery conjectured that $F(\alpha,\, T)=1+o(1)$, for $|\alpha|\ge 1,$ uniformly in compact intervals. This is the Strong Pair Correlation Conjecture, and it implies \eqref{eq:pair_correlation} by taking suitable functions $R$ in the convolution formula. The Pair Correlation Conjecture \eqref{eq:pair_correlation}, and similarly, the behaviour of $F$ in larger ranges of $\alpha$, have since proved to be deep and difficult questions, being related to important problems such as the behavior of primes in short intervals \cite{GoMo}. In this paper, we are particularly interested in the following relation, proved by Goldston \cite{Go-pair}: the Pair Correlation Conjecture \eqref{eq:pair_correlation} is equivalent to the statement
	\begin{equation}\label{eq:F_average}
	    \frac{1}{\ell}\int_b^{b+\ell}F(\alpha, \, T)\,\d \alpha \sim 1,
	\end{equation}
	as $T\to\infty$, for any fixed $b\ge 1$ and $\ell>0$. For further background on the Pair Correlation Conjecture and its equivalences, see, for instance, \cite{CCCM}, and the references therein. For a gentle introduction to the Pair Correlation Conjecture and its relation to prime numbers, see the notes \cite{Go-notes}.
	
	\subsection{Bounds via Fourier optimization}
	Recently, Carneiro, Chandee, Chirre, and Milinovich \cite{CCCM} studied these averages of $F$ over bounded intervals, by developing a general theoretical framework that relates them to some extremal problems in Fourier analysis. This was inspired by some constructions of Goldston \cite{Go} and Goldston and Gonek \cite{GG}. For example, let $\mathcal{A}_1$ be the class of continuous, even, and non-negative functions $g\in L^1(\R)$ such that $\supp \widehat{g}\subset [-1,1]$. 
	Consider the following extremal problems:\footnote{\,\,\,
	In \cite{CCCM}, the authors work with a larger class of functions instead of $\mathcal{A}_1$. See Section \ref{sec:A_LP} for further comments. Note that our class $\mathcal{A}_1$ is called $\mathcal{A}_0$ in \cite{CCCM}. We make this change of notation since, in Section \ref{sec:theoretical}, other classes $\mathcal{A}_\Delta$ will naturally appear, where the support $[-1,\,1]$ is replaced by $[-\Delta, \, \Delta]$ for a parameter $\Delta$. }
	
	\begin{problem}[EP1]
	Find 
	\begin{equation*}
	   \mathbf{C}^+:= \inf_{g\in\mathcal{A}_1\setminus\{0\}}\frac{\widehat{g}(0)+2\int_0^1 \alpha\, \widehat{g}(\alpha)\,\d \alpha}{\min_{0\le \alpha\le 1}\left\lvert \widehat{g}(\alpha)+\widehat{g}(1-\alpha)\right \rvert}.
	\end{equation*}
	\end{problem}
	\begin{problem}[EP2]
	Define the constant 
	\begin{equation}\label{eq:c0}
	    c_0:= \min_{x\in\R\setminus\{0\}}\frac{\sin x}{x}=-0.2172336282\ldots
	\end{equation}
	Find 
	\begin{equation*}
	   \mathbf{C}^-:= \sup_{\substack{g\in\mathcal{A}_1
	   \\ g(0)>0}}
	   \frac{(1-c_0)g(0) + 
	   c_0\left(\widehat{g}(0)+2\int_0^1 \alpha\, \widehat{g}(\alpha)\,\d \alpha\right)}
	   {\max_{0\le \alpha\le 1}\left(|\widehat{g}(\alpha)|+|\widehat{g}(1-\alpha)|\right)}.
	\end{equation*}
	\end{problem}
	\noindent As a consequence of their general framework, they obtain the following:
	\begin{theoremA}[{c.f. \cite[Theorem 1]{CCCM}}]
	Assume RH, let $b\ge 1$, and let $\varepsilon>0$. For sufficiently large fixed $\ell$ (possibly depending on $b$ and $\varepsilon$), as $T\to\infty$, we have 
	\begin{equation*}
	    \mathbf{C}^- -\varepsilon +o(1) \le \frac{1}{\ell}\int_b^{b+\ell}F(\alpha, \, T)\,\d \alpha \le \mathbf{C}^+ +\varepsilon + o(1).
 	\end{equation*}
	\end{theoremA}
\noindent Additionally, they establish the bounds (see \cite[Corollary 2]{CCCM} and the numerical examples in p.18 and p.20)
	\begin{equation}\label{eq:RH_old_bounds}
	    0.927818< \mathbf{C}^-\le \mathbf{C}^+ < 1.330174,
	\end{equation}
	which give the respective numerical lower and upper bounds for the left-hand side of \eqref{eq:F_average}.
	
	\subsection{$q$-analogues: an average over Dirichlet $L-$functions}
	Montgomery \cite{Mo} also suggested the investigation of the pair correlation of zeros of a family of Dirichlet $L-$functions in the $q$-aspect. One wishes to study the distribution of the low-lying zeros of $L(s,\chi)$, on average over Dirichlet characters $\chi$ (mod $q$), and over $Q\le q \le 2Q$. By taking these averages, one can obtain improvements over what is known for the Riemann zeta-function, and this provides heuristic evidence for the original case. In \cite{CLLR, Oz}, the authors obtained improvements over \eqref{eq:F_asymp} for these $q$-analogues, and used this to obtain lower bounds for the average proportion of simple zeros of Dirichlet L-functions. Later, in \cite{CCLM}, the authors introduced the idea of relating the pair correlation of zeros of $\zeta(s)$, and its $q$-analogue, to some Hilbert spaces of entire functions. Sono \cite{So} used this idea to improve the aforementioned lower bounds on the proportion of simple zeros. These were further improved in \cite{CGL}, by using a different class of functions and sophisticated numerical optimization methods (see Section \ref{sec:A_LP}). \smallskip
	
	To define these $q$-analogues, we must introduce some notation. We use the framework established in \cite{CLLR}, and follow the notation in \cite[Section 6]{CCLM}. Assume the Generalized Riemann Hypothesis for Dirichlet $L$-functions (GRH). Let $\Phi$ be a real-valued function with compact support in $(a,\, b)$, where $0<a<b$. Denote by 
	\begin{equation*}
	    \widetilde{\Phi}(s):=\int_0^\infty \Phi(x)x^{s-1}\, \d x
	\end{equation*}
	its Mellin transform. Additionally, assume that $\Phi(x^{-1})=\Phi(x)$ for all $x\in \R\setminus\{0\}$, that $\widetilde{\Phi}(it)\ge 0$ for all $t\in \R$, and that $|\widetilde{\Phi}(it)| \ll |t|^{-2}$ as $|t|\to\infty$. For instance, a possible choice satisfying all conditions (see \cite{CCLM}) is $\Phi$ such that 
	\begin{equation*}
	    \widetilde{\Phi}(s)=\left(\frac{e^s-e^{-s}}{2s}\right)^2.
	\end{equation*}
	Finally, let $W$ be a smooth, non-negative function with compact support in $(1,\,2)$. We can now define the $q$-analogue of $N(T)$ as
	\begin{equation}\label{eq:N_Phi}
	    N_\Phi(Q):= \sum_q \frac{W(q/Q)}{\phi(q)}\sideset{}{{}^*}\sum_{\chi \, (\text{mod} \, q)} \sum_{\gamma_\chi}|\widetilde{\Phi}(i\gamma_\chi)|^2,
	\end{equation}
	where the second sum (indicated by the superscript *) is over all primitive Dirichlet characters (mod $q$), and the last sum is over all non-trivial zeros $1/2+i\gamma_\chi$ of $L(s,\,\chi)$. Define the $q$-analogue of $F(\alpha, \, T)$ as
	\begin{equation}\label{def:F_phi}
	    F_\Phi(\alpha)= F_\Phi(\alpha,\, Q):= \frac{1}{N_\Phi(Q)}\sum_q \frac{W(q/Q)}{\phi(q)} \sideset{}{{}^*}\sum_{\chi \, (\text{mod} \, q)} \sum_{\gamma_\chi}|\widetilde{\Phi}(i\gamma_\chi)Q^{i\alpha\gamma_\chi} |^2.  
	\end{equation}
	
	\noindent Chandee, Lee, Liu and Radziwi\l\l \,\cite{CLLR} proved an asymptotic formula for $F_\Phi(\alpha)$ similar to \eqref{eq:F_asymp} for $|\alpha|<2$, showing, in particular, that $F_\Phi(\alpha)\sim 1$ when $1\le |\alpha|<2$ (see Lemma \ref{lem:F_phi} below for a full statement). Moreover, they conjectured that $F_\Phi(\alpha)\sim 1$ for all $|\alpha| \ge 1,$ in analogy with Montgomery's original conjecture for $F(\alpha, \, T)$. 
	We may now state our main result, which gives evidence for this conjecture.
	
	\begin{theorem}\label{thm:q-ana_asymp}
		Assume GRH, and let $b\ge 1$. For sufficiently large fixed $\ell$ (possibly depending on $b$), as $Q\to\infty$, we have 
	\begin{equation*}
	    0.982144  + o(1) < \frac{1}{\ell}\int_b^{b+\ell}F_\Phi(\alpha, \, Q)\,\d \alpha < 1.077542 + o(1).
	\end{equation*}
	\end{theorem}
	We highlight that our upper and lower bounds are very close to the conjectured value of 1. We also remark that while the size of $\ell$ in the lower bound may depend on $b$, the size of $\ell$ in the upper bound is independent of $b$. A similar situation occurs in Theorem A. For effective bounds that hold for any given $b$ and $\ell$, see Section \ref{sec:triangles}. \smallskip
	
	To prove Theorem \ref{thm:q-ana_asymp}, we develop a framework for estimating these integrals over bounded intervals via Fourier analysis, extending that of \cite{CCCM}. We take advantage of the new information available when $|\alpha|\in [1,\, 2)$, from \cite{CLLR}. This leads to slightly different Fourier extremal problems. For instance, with $\mathcal{A}_1$ and $c_0$ as above, consider the following:
	
	\begin{problem}[EP3]
	Find 
	\begin{equation*}
	   \mathbf{D}^+:= \inf_{g\in\mathcal{A}_1\setminus\{0\}}\frac{\widehat{g}(0)+8\int_0^{1/2} \alpha\, \widehat{g}(\alpha)\,\d \alpha +4\int_{1/2}^1 \widehat{g}(\alpha)\,\d \alpha}{2\min_{0\le \alpha\le 1}\left\lvert \widehat{g}(\alpha)+\widehat{g}(1-\alpha)\right \rvert}.
	\end{equation*}
	\end{problem}
	\begin{problem}[EP4]
	Find 
	\begin{equation*}
	   \mathbf{D}^-:= \sup_{\substack{g\in\mathcal{A}_1
	   \\ g(0)>0}}
	   \frac{(1-c_0)g(0) + 
	   \frac{c_0}{2}\left(\widehat{g}(0)+8\int_0^{1/2} \alpha\, \widehat{g}(\alpha)\,\d \alpha +4\int_{1/2}^1 \widehat{g}(\alpha)\,\d \alpha\right)}
	   {\max_{0\le \alpha\le 1}\left(|\widehat{g}(\alpha)|+|\widehat{g}(1-\alpha)|\right)}.
	\end{equation*}
	\end{problem}
	
	\noindent We show that $\mathbf{D}^- -\varepsilon$ and $\mathbf{D}^+ +\varepsilon$ are lower and upper bounds for the average in Theorem \ref{thm:q-ana_asymp}, respectively (see Lemma \ref{lem:asymptotic} below). The simple choice of test function $\widehat{g}(\alpha)=\max\{(1-|\alpha|),0\}$ already shows that
	\begin{equation*}
	   0.981897 < \mathbf{D}^- \ \text{ and } \  \mathbf{D}^+ < 1.083334.
	\end{equation*}
	To go further, we then numerically optimize the bounds. Note that the functionals in the above extremal problems are not smooth, due to the maximum and minimum in the denominators. Hence, we apply the principal axis method of Brent \cite{Br}, which is an algorithm for unconstrained non-smooth optimization. We also applied our optimization routine to the problems (EP1) and (EP2), and found a minor refinement in the fifth and sixth decimal digits in the bounds \eqref{eq:RH_old_bounds} from \cite[Corollary 2]{CCCM}. It seems that these are very close to the sharp values for the Fourier optimization problems. Under the hypotheses of Theorem A, we find
	\begin{equation*}
	    0.927819  + o(1) < \frac{1}{\ell}\int_b^{b+\ell}F(\alpha, \, T)\,\d \alpha < 1.330144 + o(1).
	\end{equation*}
	
	In Section \ref{sec:theoretical}, we prove a general result relating the integrals of $F_\Phi(\alpha)$ to some extremal problems, extending \cite[Theorem 7]{CCCM}. Subsequently, we use it to relate the problems (EP3) and (EP4) to Theorem \ref{thm:q-ana_asymp}. Furthermore, we provide effective bounds for the integral of $F_\Phi(\alpha)$ over any arbitrary interval, in Theorem \ref{thm:triangles}. In Section \ref{sec:numerical}, we show how to numerically optimize the bounds for (EP1)-(EP4), completing the proof of Theorem \ref{thm:q-ana_asymp}.
	
	\begin{remark}
	    Analogues of Montgomery's function $F(\alpha)$ have also been studied for other families of $L$-functions. Recently, Chandee, Klinger-Logan and Li \cite{CKL} proved an analogue of \eqref{eq:F_asymp} for an average over a family of $\Gamma_1(q)$ $L$-functions, in the range $|\alpha|<2$. Therefore, assuming GRH for this family and for Dirichlet $L$-functions, the conclusion of Theorem \ref{thm:q-ana_asymp} also holds for this family, as $q\to\infty$. See Section \ref{sec:gamma1} for more details. 
	\end{remark}
	
	\subsection{Notation}
	We denote by $\chi_A$ the characteristic function of a set $A$; $\lfloor x\rfloor$ denotes the largest integer smaller than or equal to $x$; $\lceil x \rceil$ denotes the smallest integer greater than or equal to $x$; and $\{x\}:=x-\lfloor x\rfloor$ denotes its fractional part. Additionally, $x_+:=\max\{x,\, 0\}$.
	
	\section{Fourier optimization and the average of $F_\Phi(\alpha)$}\label{sec:theoretical}
	For $\Delta\ge 1$, let $\mathcal{A}_\Delta$ be the class of continuous, even, and non-negative functions $g\in L^1(\R)$ such that $\supp \widehat{g} \subset [-\Delta, \Delta]$. For $g\in \mathcal{A}_\Delta$, denote 
	\begin{equation}
	    \rho_\Delta(g) := \widehat{g}(0)+2 \int_{0}^1 \alpha\, \widehat{g}(\alpha) \, \d \alpha + 2\int_1^{\Delta} \widehat{g}(\alpha)\, \d \alpha.  
	\end{equation}
	From the definition of $F_\Phi(\alpha)$ in \eqref{def:F_phi} and Fourier inversion, we have the convolution formula, for $R\in L^1(\R)$ with $\widehat{R}\in L^1(\R)$:
	\begin{equation}\label{eq:convolution}
	        \sum _q  \frac{W(q/Q)}{\phi(q)}\sideset{}{{}^*}\sum_{\chi \, (\text{mod} \, q)} \sum_{\gamma_\chi,\, \gamma_\chi'}
	        R\left(\frac{(\gamma_\chi-\gamma_\chi')\log Q}{2\pi}
	        \right)\widetilde{\Phi}(i\gamma_\chi)\widetilde{\Phi}(i\gamma_\chi') = 
	        N_\Phi(Q)\int_{-\infty}^\infty F_\Phi(\alpha)\, \widehat{R}(\alpha)\, \d \alpha. 
	\end{equation}
	A crucial tool is the asymptotic formula of Chandee, Lee, Liu and Radziwil\l:
	\begin{lemma}[{c.f. \cite[Theorem 1.2]{CLLR}}]\label{lem:F_phi}
	Assume GRH. Let $\varepsilon>0$. Then
	\begin{equation*}
	    F_\Phi(\alpha, \, Q) = (1+o(1))\left(f(\alpha) + \Phi(Q^{-|\alpha|})^2\log Q \left(\frac{1}{2\pi}\int_{-\infty}^\infty |\widetilde{\Phi}(ix)|^2 \, \d x
	    \right)^{-1}
	    \right)
	    + O\left(\Phi(Q^{-|\alpha|})\sqrt{f(\alpha)\log Q} \right),
	\end{equation*}
	uniformly for $|\alpha|\le 2-\varepsilon,$ as $Q\to\infty$, where $f(\alpha)=\left\{
	\begin{array}{ll}
	    |\alpha|, &\text{for \ } |\alpha|\le 1, \\
	    1, &\text{for \ } |\alpha| >1.
	\end{array}
	\right.$
	\end{lemma}
	\noindent By Plancherel's theorem for the Mellin transform, the term $\Phi(Q^{-|\alpha|})^2\log Q \left(\frac{1}{2\pi}\int_{-\infty}^\infty |\widetilde{\Phi}(ix)|^2 \, \d x
	    \right)^{-1}$ behaves like a Dirac delta at the origin (see the argument in \cite[pp. 82--83]{CLLR}). Therefore, for any fixed $1\le \Delta <2$ and $g\in \mathcal{A}_\Delta$, from \eqref{eq:convolution}, we obtain
	    \begin{equation}\label{eq:conv_rho}
	         \frac{1}{N_\Phi(Q)}\sum _q  \frac{W(q/Q)}{\phi(q)}\sideset{}{{}^*}\sum_{\chi \, (\text{mod} \, q)} \sum_{\gamma_\chi,\, \gamma_\chi'}
	        g\left(\frac{(\gamma_\chi-\gamma_\chi')\log Q}{2\pi}
	        \right)\widetilde{\Phi}(i\gamma_\chi)\widetilde{\Phi}(i\gamma_\chi')
	        =  \rho_\Delta(g) +o(1) , 
	    \end{equation}
	    as $Q\to\infty$. \smallskip
	    
	    The following problems are essentially those considered in \cite[Section 2.1.1]{CCCM}, which correspond to the case $\Delta = 1$. For any $\Delta\ge 1$, we may consider the following variations: 
	    
	    \begin{problem}[EP5]\label{prob:+}
	    Let $\ell>0$ and $\Delta\ge1$. Find 
	    \begin{equation*}
	        \mathcal{W}^+_\Delta(\ell):=\inf \sum_{j=1}^N \rho_\Delta(g_j),
	    \end{equation*}
	    where the infimum is taken over $N$ and all collections $g_1, \, g_2, \, \ldots, \, g_N\in\mathcal{A}_\Delta$ such that there exist points $\xi_1, \, \xi_2, \, \ldots, \, \xi_N\in\R$, with 
	    \begin{equation}\label{eq:boundCharac1}
	        \sum_{j=1}^N \widehat{g}_j(\alpha-\xi_j)\ge \chi_{[0,\, \ell]}(\alpha)
	    \end{equation}
	    for all $\alpha\in\R$.
	    \end{problem}
	    
	    \begin{problem}[EP6]
	    Let $\ell>0$ and $\Delta\ge1$. Find 
	    \begin{equation}\label{eq:ep6}
	        \mathcal{W}_\Delta^-(\ell):=\sup \sum_{j=1}^N (2g_j(0)- \rho_\Delta(g_j)),
	    \end{equation}
	    where the supremum is taken over $N$ and all collections $g_1, \, g_2, \, \ldots, \, g_N\in\mathcal{A}_\Delta$ such that there exist points $\xi_1, \, \xi_2, \, \ldots, \, \xi_N\in\R$, with 
	    \begin{equation}\label{eq:boundCharac2}
	        \sum_{j=1}^N \widehat{g}_j(\alpha-\xi_j)\le \chi_{[0,\, \ell]}(\alpha)
	    \end{equation}
	    for all $\alpha\in\R$.
	    \end{problem}
	    
	    \begin{problem}[EP7]
	    Let $b,\, \beta\in\R$ with $b<\beta$, and $\Delta\ge1$. Find 
	    \begin{equation}\label{eq:ep7}
	        \mathcal{W}^-_{*, \, \Delta}(b, \,\beta):=\sup \sum_{j=1}^N (g_j(0) + \tau_j( \rho_\Delta(g_j) - g_j(0))),
	    \end{equation}
	    where the supremum is taken over $N$ and all collections $g_1, \, g_2, \, \ldots, \, g_N\in\mathcal{A}_\Delta$ such that there exist points $\eta_1, \, \eta_2, \, \ldots, \, \eta_N\in\R$ and values $\tau_1, \, \tau_2, \, \ldots, \, \tau_N\in\R$, such that 
	    \begin{equation}\label{eq:boundCharac3}
	        \sum_{j=1}^N \widehat{g}_j(\alpha-\eta_j)\le \chi_{[b,\, \beta]}(\alpha)
	    \end{equation}
	    for all $\alpha\in\R$, and 
	    \begin{equation}\label{eq:tau_cond}
	        \re\left( \sum_{j=1}^N e^{2\pi i \eta_j x }g_j(x)
	        \right) \ge \sum_{j=1}^N \tau_j g_j(x),
	    \end{equation}
	    for all $x\in \R$.
	    \end{problem}
	    
	    The following result relates the problem of estimating integrals of $F_\Phi(\alpha)$ to the above problems in Fourier analysis. This general result will allow us to obtain all our bounds for these integrals. As we shall see, while the abstract formulation of these problems and the general result in Lemma \ref{lem:general} are analogous to those in \cite{CCCM}, the novelty lies in the way we may explore them, by taking advantage of the new possibilities with $1<\Delta<2$, and its interplay with the other parameters. We anticipate that, when applying Lemma \ref{lem:general}, we will usually have in mind the limit $\Delta\to2^-$.
	    
	    \begin{lemma}\label{lem:general}
	        Assume GRH, let $b\in\R$ and $\ell>0$. Let $1\le \Delta < 2.$ Then, as $Q\to\infty$, we have
	        \begin{equation}\label{eq:lem_gen}
	            \mathcal{W}_\Delta^{-}(\ell)+o(1) \le \mathcal{W}^-_{*,\, \Delta}(b,\, b+\ell) + o(1) \le \int_b^{b+\ell} F_\Phi(\alpha, \, Q) \, \d \alpha \le \mathcal{W}_\Delta^+(\ell) +o(1).
	        \end{equation}
	    \end{lemma}
	    \noindent The proof is essentially that of \cite[Theorem 7]{CCCM}, where the authors prove the analogous result for integrals of $F(\alpha, \, T)$, with $\Delta=1$. We reproduce it below, in our context, for the reader's convenience.
	    \begin{proof}
	        Assume that the bound \eqref{eq:boundCharac1} holds. We use it, combined with the convolution formula \eqref{eq:convolution} and \eqref{eq:conv_rho}, to find
	        \begin{align*}
	            &\int_b^{b+\ell} F_\Phi(\alpha) \, \d \alpha \\
	            \le& 
	            \sum_{j=1}^N \int_\R F_\Phi(\alpha)\, \widehat{g}_j(\alpha - b- \xi_j)\, \d \alpha \\
	            =& \frac{1}{N_\Phi(Q)} \sum_{j=1}^N 
	            \sum _q  \frac{W(q/Q)}{\phi(q)}\sideset{}{{}^*}\sum_{\chi \, (\text{mod} \, q)} \sum_{\gamma_\chi,\, \gamma_\chi'}
	        Q^{i(b+\xi_j)(\gamma_\chi-\gamma_\chi')} g_j\left(\frac{(\gamma_\chi-\gamma_\chi')\log Q}{2\pi}
	        \right)\widetilde{\Phi}(i\gamma_\chi)\widetilde{\Phi}(i\gamma_\chi') \\
	        \le & 
	        \frac{1}{N_\Phi(Q)} \sum_{j=1}^N 
	            \sum _q  \frac{W(q/Q)}{\phi(q)}\sideset{}{{}^*}\sum_{\chi \, (\text{mod} \, q)} \sum_{\gamma_\chi,\, \gamma_\chi'} g_j\left(\frac{(\gamma_\chi-\gamma_\chi')\log Q}{2\pi}
	        \right)\widetilde{\Phi}(i\gamma_\chi)\widetilde{\Phi}(i\gamma_\chi') \\
	        =& \sum_{j=1}^N \rho_\Delta(g_j)+o(1).
	        \end{align*}
	        This implies the upper bound in \eqref{eq:lem_gen}. To obtain the last inequality, we used the fact that $W(t)$, $g_j(t)$, and $\widetilde{\Phi}(it)$ are all non-negative (for $t\in\R$). \smallskip
	        
	        For the lower bound, we first note that $\mathcal{W}_\Delta^{-}(\ell)\le \mathcal{W}^-_{*,\, \Delta}(b,\, b+\ell)$. To see this, take a configuration that satisfies \eqref{eq:boundCharac2}.  Let $\beta=b+\ell$. Then, taking $\eta_j=\xi_j+b$, \eqref{eq:boundCharac3} is verified, and choosing $\tau_j=-1$ for all $j$, \eqref{eq:tau_cond} is also verified. With these choices, \eqref{eq:ep7} reduces to \eqref{eq:ep6}, as desired. 
	        It remains to show that $\mathcal{W}^-_{*,\, \Delta}(b,\, b+\ell) + o(1) \le \int_b^{b+\ell} F_\Phi(\alpha, \, Q) \, \d \alpha.$ Given a zero $\frac{1}{2}+i\gamma_\chi$ of $L(s,\, \chi)$ of multiplicity $m_{\gamma_\chi}$, denote $\kappa_{\gamma_\chi}:=m_{\gamma_\chi}\widetilde{\Phi}(i\gamma_\chi)^2$. Assume that \eqref{eq:boundCharac3} and \eqref{eq:tau_cond} hold. We again use them with \eqref{eq:convolution} and \eqref{eq:conv_rho} to obtain
	        \begin{align*}
	            &\int_b^{b+\ell} F_\Phi(\alpha) \, \d \alpha \\
	            \ge& 
	            \sum_{j=1}^N \int_\R F_\Phi(\alpha)\, \widehat{g}_j(\alpha - \eta_j)\, \d \alpha \\
	            =& \frac{1}{N_\Phi(Q)} \sum_{j=1}^N 
	            \sum _q  \frac{W(q/Q)}{\phi(q)}
	            \!\sideset{}{{}^*}\sum_{\chi \, (\text{mod} \, q)} \sum_{\gamma_\chi,\, \gamma_\chi'}
	        Q^{i\eta_j(\gamma_\chi-\gamma_\chi')} g_j\left(\frac{(\gamma_\chi-\gamma_\chi')\log Q}{2\pi}
	        \right)\widetilde{\Phi}(i\gamma_\chi)\widetilde{\Phi}(i\gamma_\chi') \\
	        = & 
	        \frac{1}{N_\Phi(Q)} \sum_{j=1}^N 
	            \sum _q  \frac{W(q/Q)}{\phi(q)}
	            \!\sideset{}{{}^*}\sum_{\chi \, (\text{mod} \, q)} \left\{ 
	            g_j(0) \sum_{\gamma_\chi}\kappa_{\gamma_\chi}+
	            \sum_{\gamma_\chi \neq \gamma_\chi'} 
	        Q^{i\eta_j(\gamma_\chi-\gamma_\chi')} g_j\left(\frac{(\gamma_\chi-\gamma_\chi')\log Q}{2\pi}
	        \right)\widetilde{\Phi}(i\gamma_\chi)\widetilde{\Phi}(i\gamma_\chi')  
	        \right\}\\
	        \ge & 
	        \frac{1}{N_\Phi(Q)} \sum_{j=1}^N 
	            \sum _q  \frac{W(q/Q)}{\phi(q)}
	            \!\sideset{}{{}^*}\sum_{\chi \, (\text{mod} \, q)} \left\{ 
	            g_j(0)(1-\tau_j) \sum_{\gamma_\chi}\kappa_{\gamma_\chi} +
	            \tau_j\sum_{\gamma_\chi, \,\gamma_\chi'} 
	        g_j\left(\frac{(\gamma_\chi-\gamma_\chi')\log Q}{2\pi}
	        \right)\widetilde{\Phi}(i\gamma_\chi)\widetilde{\Phi}(i\gamma_\chi')  
	        \right\}\\
	        \ge& \sum_{j=1}^N ( g_j(0) + \tau_j(\rho_\Delta(g_j)-g_j(0)) ) +o(1).
	        \end{align*}
	        This gives the desired lower bound. To obtain the last inequality, we used that, by  \eqref{eq:N_Phi},
	        \[
	        \sum _q  \frac{W(q/Q)}{\phi(q)}
	            \sideset{}{{}^*}\sum_{\chi \, (\text{mod} \, q)} \sum_{\gamma_\chi}\kappa_{\gamma_\chi}
	            \ge \sum _q  \frac{W(q/Q)}{\phi(q)}
	            \sideset{}{{}^*}\sum_{\chi \, (\text{mod} \, q)} \sum_{\gamma_\chi}\widetilde{\Phi}(i\gamma_\chi)^2 = N_\Phi(Q).
	        \]
	    \end{proof}
	    
	    \begin{remark}
	Note that if $\Delta_1\le \Delta_2$, then $\mathcal{A}_{\Delta_1}\subset \mathcal{A}_{\Delta_2}$. Therefore, $\mathcal{W}^+_\Delta(\ell)$ is non-increasing with $\Delta$, while $\mathcal{W}^-_\Delta(\ell)$ and $\mathcal{W}^-_{*, \, \Delta}(b, \, \beta)$ are non-decreasing with $\Delta$. For some properties regarding monotonicity, subadditivity, and other basic facts on the above functions $\mathcal{W}^\pm_\Delta$, we refer to \cite[Proposition 6]{CCCM}, which continues to hold for any $\Delta\ge 1$. Also, note that in the statement of Lemma \ref{lem:general}, the parameters $b$, $\ell$ and $\Delta$ are all free and independent. Here and henceforth, the error term $o(1)$ should be regarded as a function of $Q$, which may depend on all other fixed parameters ($b$, $\ell$ and $\Delta$). 
	    \end{remark}
	    
	\subsection{Triangle bounds}\label{sec:triangles}
	Here, we give simple, effective bounds for the integral of $F_\Phi(\alpha)$ over an arbitrary interval, by using (EP5) and (EP6) with the functions $\widehat{g}_j$ chosen as triangles. Our bounds have the property of being continuous and non-decreasing with $\ell.$ To begin, let $\Delta\ge 1$. For $0<\delta \le \Delta$, let
	\begin{equation}
	    K_\delta(x)= \delta\left(\frac{\sin \pi \delta x}{\pi\delta x} \right)^2\ \text{ and }\ \widehat{K_\delta}(\alpha)=\left( 1-\frac{|\alpha|}{\delta} \right)_+.
	\end{equation}
	Note that
	\begin{equation}\label{eq:rho_k_delta}
	    \rho_\Delta (K_\delta) = \left\{
	    \begin{array}{ll}
	    1+\frac{\delta^2}{3}, &\text{if \ } 0<\delta \le 1, \\
	    \delta + \frac{1}{3\delta}, &\text{if \ } 1<\delta \le \Delta.
	\end{array}
	    \right.
	\end{equation}
	\begin{figure}[t] 
   \centering
   \includegraphics[width=3in]{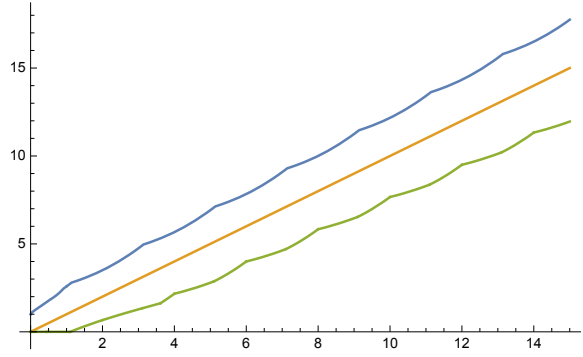}
   \label{fig:bounds}\caption{The upper bound (in blue) and the lower bound (in green) given in Theorem \ref{thm:triangles}, compared with the q-analogue of the Pair Correlation Conjecture (in yellow).}
\end{figure}
	\begin{theorem}\label{thm:triangles}
	Assume GRH, let $b\ge 1$, and let $\ell>0$. Then, as $Q\to\infty$, we have 
	\begin{equation*}
	    \mathcal{C}^-(\ell) +o(1) \le  \int_b^{b+\ell} F_\Phi(\alpha, \, Q) \, \d \alpha \le  \mathcal{C}^+(\ell) +o(1),
	\end{equation*}
	where
	%
	\begin{equation}\label{eq:triangle_upper}
	    \mathcal{C}^+(\ell)  = 
	    \left\{
	    \begin{array}{l}
	        \!\!\frac{13(\ell+2) }{12}
	         +\frac{1}{3}\left\{\frac{\ell}{2}\right\}^3 -\frac{7 }{6}\left\{\frac{\ell}{2}\right\}
	         -\frac{1}{6}\left(2\left\{\frac{\ell}{2}\right\}^3-6\left\{\frac{\ell}{2}\right\}^2-6\left\{\frac{\ell}{2}\right\}+5\right)_+, \ \ \ 
	    \rm{if }\  \ell \ge 1,  \\
	    \!\!\min\left \{\!
	    \frac{13 (\ell+2)}{12}
	         +\frac{1}{3}\left\{\frac{\ell}{2}\right\}^3 -\frac{7 }{6}\left\{\frac{\ell}{2}\right\}
	         -\frac{1}{6}\left(2\left\{\frac{\ell}{2}\right\}^3-6\left\{\frac{\ell}{2}\right\}^2-6\left\{\frac{\ell}{2}\right\}+5\right)_+ ;
	    \,\,(1+c)\left(1+\frac{\ell^2(1+c)^2}{12c^2}
	    \right)
	    \right\}\, , \\ 
	 	 \hspace{3.5cm}\rm{if  }\  
	  0<\ell <1, \ \rm{ with } \ c = \max\left \{6^{-1/3}\ell^{2/3}\, ; \,\frac{\ell}{2-\ell} \right\};
	\end{array}
	    \right.
	\end{equation}
	and 
	%
	\begin{equation}
	    \mathcal{C}^-(\ell)  = 
	    \left\{
	    \begin{array}{l}
	        \max\left\{
	        \frac{11 (\ell-2)}{12}+
	        
	        \frac{1}{2} \left\{\frac{\ell}{2} \right\}^2-\frac{5}{6} \left\{\frac{\ell}{2} \right\}+\frac{1}{3}+
	        \frac{1}{6}\left\{\frac{\ell}{2} \right\}  
	        \left(-\left\{\frac{\ell}{2} \right\}^2+ 6\left\{\frac{\ell}{2} \right\}-3\right)_+
	        \, \, ;\ \  
	        \frac{1}{2}\ell-\frac{2}{3\ell}
	        \right\}\,\, , \,\,\,\,
	    \rm{if }\  \ell \ge 2,  \\
	    \left (
	    \ell-1-\frac{\ell^2}{12}
	    \right)_+, 
	    \hspace{2cm}
	 	 \rm{if  }\  
	  0<\ell \le 2.
	\end{array}
	    \right.
	\end{equation}
	\end{theorem}
	%
	\begin{proof}
	We would like to apply Lemma \ref{lem:general} with $\Delta\to2.$ To achieve this, we must obtain continuous bounds for an arbitrary $\Delta\in(1, \, 2)$. For simplicity, we will additionally assume that $\frac{4}{3}\le \Delta <2$ throughout the proof.\vspace{0.3cm}
	
	\noindent \textit{Upper bound.}
	Following the strategy in \cite{CCCM}, we choose $n\ge 0$ large triangles, with two additional small triangles at the beginning and end. In (EP5), we take $N=n+2$, $\widehat{g}_j= K_\Delta$ for $2\le j\le n+1$, and $\widehat{g}_1=\widehat{g}_{N+2}=\frac{\delta}{\Delta}\widehat{K}_\delta$. These are the similar triangles with base $2\Delta$ and height 1, and base $2\delta$ and height $\frac{\delta}{\Delta}$, respectively. Moreover, let
	\begin{equation}\label{eq:l+}
	    \Delta(n-1)+2\delta =\ell, 
	\end{equation}
	and consider the translates given by $\xi_1=0$, $\xi_j = \delta +\Delta(j-2)$ for $2\le j \le n+1$, and $\xi_{n+2}=\Delta(n-1)+2\delta=\ell$. Then, condition \eqref{eq:boundCharac1} is satisfied. We must now choose $n$ and $\delta$ in terms of $\ell$ and $\Delta$, such that $\eqref{eq:l+}$ holds. If $\frac{\ell}{\Delta}\in \N$, we may take $(n\, , \delta) = (\frac{\ell}{\Delta}, \, \frac{\Delta}{2})$. This gives the upper bound 
	\[
	\left(1+\frac{1}{3\Delta^2}\right)\ell + \frac{\Delta^2}{12} +1.
	\]
	If $\frac{\ell}{\Delta}\notin \N$, we have the choices  $(n\, , \delta) = \left(\lfloor \frac{\ell}{\Delta}\rfloor+1, \, \frac{\Delta}{2}\left\{\frac{\ell}{\Delta}\right\}\right)$ or $\left(\lfloor \frac{\ell}{\Delta}\rfloor, \, \frac{\Delta}{2}+\frac{\Delta}{2}\left\{\frac{\ell}{\Delta}\right\}\right)$. Note that the first choice implies $0<\delta < \frac{\Delta}{2}<1$, while the second choice implies $\frac{\Delta}{2}<\delta<\Delta <2.$ We take the minimum of both possibilities, and we must further divide the second choice in cases, depending on whether or not $\delta \ge 1$, to apply \eqref{eq:rho_k_delta}. Note that $\delta\ge 1$ if and only if $\left\{\frac{\ell}{\Delta}\right\}\ge \frac{2}{\Delta}-1$. This yields the upper bound $\mathcal{W}_\Delta^+(\ell)\le C^+_\Delta(\ell)$, where 
	\begin{equation}\label{eq:tri_upper_delta}
	    C^+_\Delta(\ell)  = \left\{
	    \begin{array}{ll}
	    \left(\frac{1}{3 \Delta^2}+1\right) (\Delta+\ell)+p_\Delta(\{\ell/\Delta\}), &\text{if \ }
	    \left\{\frac{\ell}{\Delta}\right\} <\frac{2}{\Delta} -1, \\
	    \left(\frac{1}{3 \Delta^2}+1\right) (\Delta+\ell)
	 +
	 q_\Delta(\{\ell/\Delta\})
	 -r_\Delta(\{\ell/\Delta\})_+, &\text{if \ } 
	  \left\{\frac{\ell}{\Delta}\right\} \ge \frac{2}{\Delta} -1;
	\end{array}
	    \right.
	\end{equation}
	 \[
	 p_{\Delta}(x):=\frac{(x+1) \left(\Delta^3 (x+1)^2-12 \Delta^2+12 \Delta-4\right)}{12 \Delta},
	 \ \ \ \  
	 q_\Delta(x):= \frac{\Delta^2 x^3}{12}+x \left(1-\Delta-\frac{1}{3\Delta}\right)
	 \]
	 and
	 \[
	 r_\Delta(x):= \frac{\Delta^2 x^3}{12}-\frac{\Delta x^2}{2}+(1-\Delta) x+\frac{\Delta}{2}-\frac{1}{3 \Delta}.
	 \]
	 
	 One can verify that, for all $1\le \Delta \le 2$, $r_\Delta(x)$ has a unique root in the interval $(0, \, 1)$, and, if $\Delta\ge \frac{4}{3}$, this root is always greater than $\frac{2}{\Delta}-1$, since $r_\Delta\left(\frac{2}{\Delta}-1\right)>0$ and $r_\Delta(1)<0$. This root denotes the transition between the two choices of $n$ and $\delta$ above. In particular, for $\Delta\ge 4/3$ and $\ell>0$, note that $C^+_\Delta(\ell)$ is a continuous function of $\ell$ and $\Delta$. Therefore, for fixed $\ell$, and separately analyzing the cases $\frac{\ell}{2}\in \N$ and $\frac{\ell}{2}\notin \N$, we may let $\Delta\to2$ in \eqref{eq:tri_upper_delta} and Lemma \ref{lem:general} to obtain the upper bound in Theorem \ref{thm:triangles}, in the case $\ell\ge 1$. The upper bound for $0<\ell <1$ follows from taking $\Delta=1$ in Lemma \ref{lem:general}, and applying directly the bounds for $\mathcal{W}^+_1(\ell)$ in \cite[Theorem 9]{CCCM}.\vspace{0.3cm}
	 
	 \noindent \textit{Lower bound. }
	 If $0<\ell\le 2\Delta$, we may take the single triangle $\widehat{g}_1=\widehat{K}_{\ell/2}$, with $\xi_1=\ell/2$. This gives the lower bound $\mathcal{W}_\Delta^-(\ell)\ge (\ell-1-\frac{\ell^2}{12})_+$ if $\ell\le 2$ (where we have the trivial bound of zero for $\ell\le 6-2\sqrt{6}=1.101\ldots$), and $\mathcal{W}_\Delta^-(\ell)\ge \frac{1}{2}\ell - \frac{2}{3\ell}$ if $2<\ell\le 2\Delta$. Letting $\Delta\to2$, we obtain these same lower bounds in Lemma \ref{lem:general} for any $0<\ell< 4$. 
	 \begin{figure}[t] 
   \centering
   \includegraphics[width=3in]{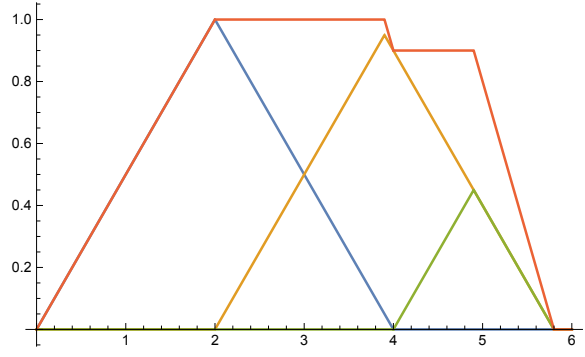}
   \label{fig:triangles}\caption{A superposition of triangles of three different sizes gives a minorant of $\chi_{[0,\,\ell]}$, producing a continuous lower bound. This is the construction for $\ell=5.8$, when $\Delta \to 2^-.$}
\end{figure}
	 
	 \vspace{0.3cm}
	 In the next step, we must diverge slightly from the strategy in \cite{CCCM} to obtain bounds that are continuous, with respect to $\ell$ and $\Delta$. For any $\ell \ge \Delta$, we combine triangles of three different sizes, instead of two as before. First, we take $n\ge 0$ big triangles, followed by one medium triangle, and possibly one last small triangle. 
	 Let $n=\left \lfloor\frac{\ell}{\Delta} \right\rfloor-1 $, $\delta_1 = \frac{\Delta}{2}\left(1+\left\{\frac{\ell}{\Delta}\right\}\right)$, and $\delta_2=\frac{\Delta}{2}\left\{\frac{\ell}{\Delta}\right\}$. Consider the functions $\widehat{g}_j=\widehat{K_2}$ for $1\le j \le n$, $\widehat{g_{n+1}}=\frac{\delta_1}{\Delta}\widehat{K_{\delta_1}}$, and $\widehat{g_{n+2}}=\frac{\delta_2}{\Delta}\widehat{K_{\delta_2}}$. Take $\xi_j=\Delta j$ for $1\le j\le n$, $\xi_{n+1}=\Delta n+\delta_1$, and $\xi_{n+2} = \Delta(n+1)+\delta_2$. The last pair $(\widehat{g_{n+2}}, \, \xi_{n+2})$ is only included when $2K_{\delta_2}(0)-\rho(K_{\delta_2})>0$, that is, when $\left\{\frac{\ell}{\Delta}\right\}>\frac{6-2\sqrt{6}}{\Delta}$. Note that, when $\Delta\ge \frac{4}{3},$ we have $\delta_1\ge \frac{\Delta}{2}\ge \frac{2}{3}$, so that $2K_{\delta_1}(0)-\rho(K_{\delta_1})$ is always positive in this range.  Additionally, note that $\ell = \Delta n +2\delta_1= \Delta(n+1)+2\delta_2$, and \eqref{eq:boundCharac2} is satisfied. 
	 
	\noindent The above configuration, in (EP6), yields $\mathcal{W}_\Delta^-(\ell)\ge C^-_\Delta(\ell)$, where 
		\begin{equation}\label{eq:tri_lower_delta}
	    C^-_\Delta(\ell)  = \left\{
	    \begin{array}{ll}
	    \left(\frac{1}{3 \Delta^2}+1\right) (\Delta+\ell)+u_\Delta(\{\ell/\Delta\}), &\text{if \ }
	    \left\{\frac{\ell}{\Delta}\right\} <\frac{2}{\Delta} -1, \\
	    \left(\frac{1}{3 \Delta^2}+1\right) (\Delta+\ell)
	 +
	 v_\Delta(\{\ell/\Delta\})+w_\Delta(\{\ell/\Delta\})_+, &\text{if \ } 
	  \left\{\frac{\ell}{\Delta}\right\} \ge \frac{2}{\Delta} -1;
	\end{array}
	    \right.
	\end{equation}
	\begin{equation*}
	    u_\Delta(x) = \frac{-\frac{1}{4} \Delta^3 (x+1)^3+3 \Delta^2 \left(x^2+1\right)-3 \Delta (x+1)+2 x}{6 \Delta};
	\end{equation*}
	\begin{equation*}
	    v_\Delta(x) = \frac{(x-1) \left(3 \Delta^2 (x-1)+4\right)}{12 \Delta};
	    \ \ \text{ and } \ \ 
	    w_\Delta(x) = \frac{6 \Delta x  \left(-\frac{1}{12} \Delta^2 x^2+\Delta x-1\right)}{12 \Delta}.
	\end{equation*}
	Note that, for $0\le x\le 1$ and $1\le \Delta \le 2$, $w_\Delta(x)>0$ if and only if $x>\frac{6-2\sqrt{6}}{\Delta}$. Moreover, we have that $0\le \frac{2}{\Delta}-1<\frac{6-2\sqrt{6}}{\Delta} <1$. In particular, $C^-_\Delta(\ell)$ is a continuous function of $\ell$ and $\Delta$, and we may take $\Delta\to 2$ to obtain the lower bounds in Theorem \ref{thm:triangles}. In the lower bound for $\ell>2$, the maximum is attained by the second function for $2<\ell\le \ell_1$, and by the first function for $\ell >\ell_1$, where $\ell_1=3.609\ldots$
	%
	\end{proof}

	 

	\begin{remark}
	    An important technical feature of the bounds in Theorem \ref{thm:triangles} is their continuity, which helps to take $\Delta\to 2$. To achieve this continuity for an arbitrary $\Delta\in (4/3,\, 2)$, we must take precise configurations of triangles, slightly different from those considered in \cite{CCCM}, and take care with the cases that arise depending on the size of $\Delta$. As $\ell\to\infty$, it is clearly convenient to take $\Delta$ as large as possible, as we have the multiplying factors $(1\pm\frac{1}{3\Delta^2})$. However, for some fixed values of $\ell$, one could do slightly better than stated in Theorem \ref{thm:triangles}, by using the general bounds in \eqref{eq:tri_upper_delta} and \eqref{eq:tri_lower_delta} and choosing an optimal $\Delta$ in $(4/3, \, 2)$.
	\end{remark} 
	\subsection{Asymptotic bounds}
	Recall that 
	\begin{equation*}
	   \mathbf{D}^+:= \inf_{g\in\mathcal{A}_1\setminus\{0\}}\frac{\widehat{g}(0)+8\int_0^{1/2} \alpha\, \widehat{g}(\alpha)\,\d \alpha +4\int_{1/2}^1 \widehat{g}(\alpha)\,\d \alpha}{2\min_{0\le \alpha\le 1}\left\lvert \widehat{g}(\alpha)+\widehat{g}(1-\alpha)\right \rvert}
	\end{equation*}
	and 
\begin{equation*}
	   \mathbf{D}^-:= \sup_{\substack{g\in\mathcal{A}_1
	   \\ g(0)>0}}
	   \frac{(1-c_0)g(0) + 
	   \frac{c_0}{2}\left[\widehat{g}(0)+8\int_0^{1/2} \alpha\, \widehat{g}(\alpha)\,\d \alpha +4\int_{1/2}^1 \widehat{g}(\alpha)\,\d \alpha\right]}
	   {\max_{0\le \alpha\le 1}\left(|\widehat{g}(\alpha)|+|\widehat{g}(1-\alpha)|\right)}.
	\end{equation*}	
	In this section, we begin the proof of Theorem \ref{thm:q-ana_asymp}, by connecting the integrals of $F_\Phi(\alpha)$ to the above extremal problems. The main idea is to consider, in (EP5) and (EP7), copies of a single function $\widehat{g}$, instead of a triangle, so that we may then optimize over admissible functions.
	\begin{lemma}\label{lem:asymptotic}
	    Assume GRH, and let $b\ge 1$. For sufficiently large fixed $\ell$, as $Q\to\infty$, we have 
	\begin{equation*}
	    \mathbf{D}^- -\varepsilon  + o(1) < \frac{1}{\ell}\int_b^{b+\ell}F_\Phi(\alpha, \, Q)\,\d \alpha < \mathbf{D}^+ +\varepsilon  + o(1).
	\end{equation*}
	\end{lemma}
	
	\noindent\textit{Proof.} Throughout the proof, let $g\in\mathcal{A}_1$, so that $\supp \widehat{g}\subset[-1,1]$. For $1< \Delta< 2$, consider the dilation $g_\Delta(\alpha)$=$\Delta g(\Delta \alpha)$, so that $g_\Delta\in \mathcal{A}_\Delta$. Again, we must first obtain bounds for an arbitrary $\Delta\in (1, \, 2)$, with the goal of taking $\Delta\to 2^-$ at the end.
	
	\subsubsection{Upper bound}
	Assume that $\min_{0\le \alpha\le 1}\left\lvert \widehat{g}(\alpha)+\widehat{g}(1-\alpha)\right\rvert \neq 0 $. Then, since $\widehat{g}(0)>0$ and $\widehat{g}$ is continuous, we must have $\widehat{g}(\alpha)+\widehat{g}(1-\alpha)>0$ for all $\alpha\in [0,\, 1]$, and by multiplying by an appropriate constant, we may assume that 
	\begin{equation}\label{eq:min1}
	\min_{0\le \alpha\le 1} \left(\widehat{g}(\alpha)+\widehat{g}(1-\alpha)\right) = 1.
	\end{equation}
	
	 \noindent In (EP5), given $\ell$, let $N=\left\lceil\frac{\ell}{\Delta}\right\rceil +1$. Consider the $N$ functions $\widehat{g_j}=\widehat{g_\Delta}$, for $1\le j\le N$, and take the translates $\xi_j=\Delta(j-1)$. Then, by \eqref{eq:min1}, the fact that $\supp \widehat{g}\in [-1,1]$, and that $\widehat{g}$ is even, we obtain
	\begin{equation}\label{eq:sum_g}
	\begin{split}
	    \sum_{j=1}^N\widehat{g_j}(\alpha-\xi_j) &= \sum_{j=1}^N \widehat{g}\left(\frac{\alpha}{\Delta}-j+1\right)\\
	    &= \widehat{g}\left(\left\{\frac{\alpha}{\Delta}\right\}\right)
	    + \widehat{g}\left(1-\left\{\frac{\alpha}{\Delta}\right\}\right)\\
	    &\ge 1,
	\end{split}
	\end{equation}
	for all $0\le \alpha \le \Delta \left\lceil\frac{\ell}{\Delta}\right\rceil$, in particular for $0\le \alpha \le \ell$. By an argument of Carneiro, Chandee, Chirre, and Milinovich (adding a finite number of triangles if necessary, see \cite[p. 18]{CCCM}) we may assume that this sum is non-negative for all $\alpha$, and therefore \eqref{eq:boundCharac1} is satisfied. This gives the bound
	\[
	\mathcal{W^+}_\Delta(\ell) \le \frac{\rho_\Delta (g_\Delta)}{\Delta}\, \ell + O(1),
	\]
	where the implied constant may depend on $g$, but not on $\ell$ or $\Delta$. Note that the function 
	\[\Delta \mapsto  \frac{\rho_\Delta(g_\Delta)}{\Delta} =
	\frac{1}{\Delta}\left(\widehat{g}(0)+2\Delta^2\int_0^\frac{1}{\Delta} \alpha\, \widehat{g}(\alpha)\,\d \alpha +2\Delta\int_\frac{1}{\Delta}^1 \widehat{g}(\alpha)\,\d \alpha\right)
	\]
	is continuous for $1\le \Delta \le 2.$ Then, we may take $\Delta\to 2$ in Lemma \ref{lem:general} with the above bound, to obtain, for any fixed $\varepsilon>0$ and $\ell$ sufficiently large, 
	\[
	\int_b^{b+\ell} F_\Phi(\alpha, \, Q)\, \d \alpha \le \ell\, (\mathbf{D}^+  +\varepsilon) + o(1),
	\]
	as $Q\to\infty$. This proves the upper bound in Lemma \ref{lem:asymptotic}.
	\subsubsection{Lower bound}\label{sec:asymptotic-lower}
	Here, we will use the framework of (EP7). We may assume, without loss of generality, that $g(0)>0$, and that $\max_{0\le \alpha\le 1} |\widehat{g}(\alpha)|+|\widehat{g}(1-\alpha)|=1$. For a fixed $b\ge1$ and large $\ell$, let $\beta = b + \ell$, and write
	\begin{align}\label{eq:decomposing}	
	\int_b^{b+\ell} F_\Phi(\alpha, \, Q)\, \d \alpha &= 	1/2\int_{-\beta}^{\beta}F_\Phi(\alpha, \, Q)\, \d \alpha -  
	\int_{0}^{b}F_\Phi(\alpha, \, Q)\, \d \alpha 
.
	\end{align}
	Let $n=\left\lfloor\frac{\beta}{\Delta}\right\rfloor$. In (EP7), let $N=2n-1$, and take $\widehat{g_j}=\widehat{g_\Delta}$ and $\eta_j=\Delta(n-j)$, for $1\le j\le N$. Define
	\[
	\mathfrak{m}(n)=\min_{x\in \R} D_n(x),
	\]
	where
	\[D_n(x)=\sum_{k=-n}^ne^{ikx} = \frac{\sin((n+1/2)x)}{\sin(x/2)}
	\]
	is the Dirichlet kernel. From \cite[Equation (2.37)]{CCCM}, it is known that 
	\begin{equation}\label{eq:c0_limit}
	    \lim_{n\to\infty} \frac{\mathfrak{m}(n)}{n} = 2c_0, \ \textrm{ and moreover, }\ \left|\frac{\mathfrak{m}(n)}{n} - 2c_0\right| \ll \frac{1}{n},
	\end{equation}
	where $c_0$ is defined in \eqref{eq:c0}.
	Let 
	\[
	\tau_j = \inf_{
	\substack{x\in\R \\
	g(x)\neq 0}
	} \frac{g(x)\re \left(\sum_{j=1}^{2n-1} e^{2\pi i \Delta(n-j)x} 
	\right)}{(2n-1)g(x)
	}
	= \frac{\mathfrak{m}(n-1)}{2n-1}.
	\]
	Then, \eqref{eq:tau_cond} is automatically satisfied, and we can verify \eqref{eq:boundCharac3} (where $b=-\beta$). 
	This gives the bound
	\[\mathcal{W}_{*, \, \Delta}^-(-\beta, \, \beta)\ge 
	(2n-1)\left(\Delta g(0) -\frac{\mathfrak{m}(n-1)}{2n-1}(\rho_\Delta(g_\Delta) - \Delta g(0))\right).
	\]
	This implies, by \eqref{eq:c0_limit}, that
	\[\frac{\mathcal{W}_{*, \, \Delta}^-(-\beta, \, \beta)}{2}\ge
	\frac{\beta}{\Delta}\, (\Delta g(0)(1-c_0) +c_0\, \rho_\Delta (g_\Delta) ) -O(1),
	\]
	where the implied constant may depend on $g$ but not on $\Delta$ or $\beta$ (note that $\rho_\Delta(g_\Delta)$ can be bounded in terms of $g$, uniformly in $\Delta$). Now, we apply Lemma \ref{lem:general} and \eqref{eq:decomposing}, and let $\Delta\to 2$ as before. For any fixed $\varepsilon>0$ and $b\ge 1$, we obtain, for sufficiently large fixed $\ell$, that
	\[	\int_b^{b+\ell} F_\Phi(\alpha, \, Q)\, \d \alpha \ge \ell\, (\mathbf{D^-}-\varepsilon) +o(1), 
	\]
	as $Q\to\infty,$ as desired. \qed
\subsection{$\Gamma_1(q)$-analogues: an average over automorphic $L$-functions}\label{sec:gamma1}
	In this section, for the convenience of the reader, we briefly define the $\Gamma_1(q)$-analogue of $F(\alpha)$, and show how it also satisfies the conclusions of Theorem \ref{thm:q-ana_asymp} and Theorem \ref{thm:triangles}. This is the framework of Chandee, Klinger-Logan and Li in \cite{CKL}, to which we refer for more details. 
	The authors consider a large family of $GL(2)$ $L$-functions, as follows. Let $k$ and $q$ be positive integers, with $k\ge 3$. Consider the subgroups of $GL_2(\Z)$
	\[\Gamma_0(q)=\left\{ 
	\begin{pmatrix}
a & b \\
c & d
\end{pmatrix}:\, ad-bc=1, \, c\equiv 0\, (\text{mod}\, q)   
	\right\} ,\ \text{and}\  \Gamma_1(q)=\left\{ 
	\begin{pmatrix}
a & b \\
c & d
\end{pmatrix}\in \Gamma_0(q): \, a\equiv d\equiv 1\, (\text{mod} \, q )
	\right\}.
	\]
	Let $S_k(\Gamma_0(q), \, \chi)$ be the space of cusp forms of weight $k\ge 3$ for $\Gamma_0(q)$ and nebentypus character $\chi$ (mod $q$). Let $\mathcal{H}_\chi\subset S_k(\Gamma_0(q), \, \chi)$ be an orthogonal basis of $S_k(\Gamma_0(q),\, \chi)$ consisting of Hecke cusp forms, normalized so that the first Fourier coefficient is 1. It is known that each $f\in \mathcal{H}_\chi$ has an associated $L$-function $L(s,\, f)$. Assume GRH for all the $L(s,\, f)$ and for all Dirichlet $L$-functions. Then, we define the $\Gamma_1(q)$-analogue of $F(\alpha)$ as
	\[
	F_\Phi^*(\alpha, \, q):= \frac{2\,\Gamma(k-1)}{N_\Phi^*(q)\, \phi(q)\, (4\pi)^{k-1}} \sum_{\substack{\chi \ (\text{mod}\ q)\\
	\chi(-1)=(-1)^k
	}
	}\sum_{f\in\mathcal{H}_\chi}\frac{1}{\|f\|^2}
	\left|\sum_{\gamma_f} \widetilde{\Phi}(i\gamma_f)\,q^{i\gamma_f\alpha}
	\right|^2, 
	\]
	where
	\[
	N_\Phi^*(q):= \frac{2\,\Gamma(k-1)}{\phi(q)\, (4\pi)^{k-1}} \sum_{\substack{\chi \ (\text{mod}\ q)\\
	\chi(-1)=(-1)^k
	}
	}\sum_{f\in\mathcal{H}_\chi}\frac{1}{\|f\|^2}\sum_{\gamma_f}\left|\widetilde{\Phi}(i\gamma_f)
	\right|^2,
	\]
	and the inner sums run over the ordinates of all non-trivial zeros $\frac{1}{2}+i\gamma_f$ of $L(s,\, f)$. Note that 
	\begin{equation*}
	    S_k(\Gamma_1(q)) = \bigoplus_{\chi\,(\text{mod}\ q) }S_k(\Gamma_0(q),\chi),
	\end{equation*}
	where $S_k(\Gamma_1(q))$ is the space of holomorphic cusp forms for $\Gamma_1(q)$. Therefore, we may think of $F_\Phi^*(\alpha, \, q)$ as the $\Gamma_1(q)$-analogue of $F(\alpha).$ \smallskip
	
	In \cite[Theorem 1.1]{CKL}, the authors show that the same asymptotic formula in Lemma \ref{lem:F_phi} holds, with $F_\Phi^*(\alpha, \, q)$ replacing $F_\Phi(\alpha, \, Q)$, as $q\to\infty$. Fourier inversion yields, as in \eqref{eq:convolution}, 
	\[\frac{2\,\Gamma(k-1)}{\phi(q)\, (4\pi)^{k-1}} \!\!\sum_{\substack{\chi \ (\text{mod}\ q)\\
	\chi(-1)=(-1)^k
	}
	}\sum_{f\in\mathcal{H}_\chi}\frac{1}{\|f\|^2}
	\sum_{\gamma_f,\, \gamma_f'}
	R\left(\frac{(\gamma_f-\gamma_f')\log q}{2\pi}\right)
	\widetilde{\Phi}(i\gamma_f)\widetilde{\Phi}(i\gamma_f')
	= N_\Phi^*(q)\int_\R F^*_\Phi(\alpha,\,q)\, \widehat{R}(\alpha)\,\d \alpha.
	\]
	Then, the same argument in the proof of Lemma \ref{lem:general} shows that, for any fixed $b\ge 1$, $\ell >b$, and $1\le\Delta<2$, 
	\[\mathcal{W}^-_{*,\, \Delta}(b,\, b+\ell) + o(1) \le \int_b^{b+\ell} F_\Phi^*(\alpha, \, q) \, \d \alpha \le \mathcal{W}_\Delta^+(\ell) +o(1),\] as $q\to \infty$. The bounds for $\mathcal{W}_\Delta^\pm(\ell)$ and $\mathcal{W}_{*,\, \Delta}^-(\ell)$, given in the proofs of Theorem \ref{thm:q-ana_asymp} and Theorem \ref{thm:triangles}, now immediately imply the analogous theorems for $F^*_\Phi(\alpha,\, q)$, with the same constants.
	%
	\section{Numerically optimizing the bounds}\label{sec:numerical}
We must optimize the functionals given in (EP1)-(EP4), over functions in the class $\mathcal{A}_1$. First, we transform these optimization problems over $\mathcal{A}_1$ into unrestricted optimization problems over $\R^{d+1}$, where $d\in \N$. By a result of Krein \cite[p. 154]{Ach}, if $g\in \mathcal{A}_1$, then $g(x)=|h(x)|^2$, for some $h\in L^2(\R)$ with $\supp \widehat{h}\subset [-\frac{1}{2}, \frac{1}{2}]$. We may then search over functions of the form $\widehat{h}(x)=p(x)\chi_{[-\frac{1}{2}, \frac{1}{2}]}$, where 
\[p(x):=\sum_{i=0}^d a_ix^i\]
is a polynomial of degree $d$. The numerators in (EP1)-(EP4) are now bilinear forms 
\[\sum_{i,\, j = 0}^d c_{ij}a_ia_j
\]
in the coefficients of $p$. To implement these bilinear forms, one may compute the values of $c_{ij}$ by numerically evaluating the numerators of the functionals on the polynomials $p_{ij}(x):= x^i + x^j$, for $0\le i, \, j\le d$. The maxima and minima in the denominators in (EP1)-(EP4) may be computed via a simple 1-dimensional optimization routine.

We proceed to optimize over the coefficients $a_i$ via the principal axis method of Brent \cite{Br}. This is an iterative algorithm without derivatives, which requires two initial values for all coefficients $a_i$. We take $d\le 12$, run the algorithm with many different randomly-chosen initializations, and additionally run it with initializations found previously from running this procedure with lower degrees. In this way, we found the following functions:
\[p_1(x) = 200 x^{12}+815 x^{10}-152 x^8-59 x^6+\frac{69 x^4}{10}-\frac{157 x^2}{1000}+1,
\]
which shows $\mathbf{D}^+ < 1.077542$ in (EP3); and
\[
p_2(x)= -x^2+5,
\]
which shows $\mathbf{D}^- > 0.982144$ in (EP4). This proves\footnote
{\,\,\,The numerators in the functionals can be computed exactly in rational arithmetic, in terms of $c_0$. The maximum and minimum in the denominators, and the value of $c_0$ given in \eqref{eq:c0}, may be easily verified to the desired precision, for instance by first isolating the critical points and then applying the bisection method, using interval arithmetic. }
Theorem \ref{thm:q-ana_asymp}. We also found
\[p_3(x) = -3855 x^{12}+2203 x^{10}-\frac{2743 x^8}{10}-\frac{152 x^6}{5}+\frac{303 x^4}{100}-\frac{7 x^2}{250}+1,
\]
which shows $\mathbf{C}^+ < 1.330144$ in (EP1); and
\[p_2(x) = -x^2 +\frac{250}{47},
\]
which shows $\mathbf{C}^- > 0.927819$ in (EP2). We also ran this routine with $d=14$, and found no improvement in the first six decimal digits with respect to the above functions.
\begin{figure}[t] 
   \centering
   \includegraphics[width=3in]{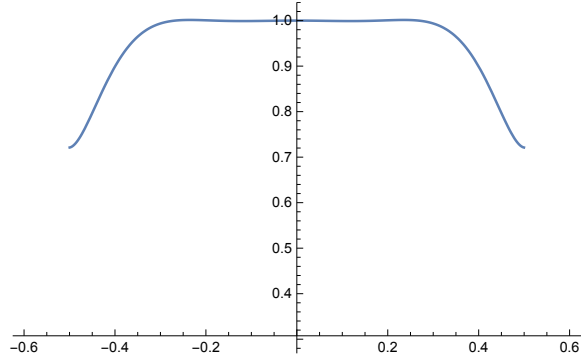}
   \label{fig:p1}\caption{The function $\widehat{h}(x)=p_1(x)\chi_{[-1/2,\,1/2]}$ is a perturbation of $\chi_{[-1/2,\,1/2]}$.}
\end{figure}

	\subsection{Remarks on a larger class of functions}\label{sec:A_LP}
	Let $\mathcal{A}$ be the class of continuous, even, and non-negative functions $g\in L^1(\R)$, such that $\widehat{g}(\alpha)\le 0$ for $|\alpha|\ge 1$. Note that $\mathcal{A}_1\subset\mathcal{A}$. Cohn and Elkies \cite{CE} first used this class $\mathcal{A}$ to obtain upper bounds for the sphere packing problem. Recently, Chirre, Gon\c{c}alves, and de Laat \cite{CGL} also used it to sharpen bounds in the theory of the Riemann zeta and other $L$-functions. With this more general framework, the problems considered in \cite{CGL} are reduced to convex optimization problems, which the authors solve numerically via semidefinite programming (see \cite{semi} for background on semidefinite programming). Furthermore, Chirre, Pereira, and de Laat \cite{CPL} used a similar framework, with semidefinite programming, to obtain fine estimates for primes in arithmetic progressions, following a Fourier optimization approach by Carneiro, Milinovich, and Soundararajan \cite{CMS}. In all these works, the authors use these numerical techniques to construct test functions of the form $g(x)=p(x)e^{-\pi x^2}$, where $p$ is a polynomial of a certain degree. 
	
	In \cite{CCCM}, the authors also build their theoretical framework using this larger class $\mathcal{A}$, while working with the simpler class $\mathcal{A}_1$ to obtain their bounds. We explored the optimizations problems with this larger class, with the purpose of refining Theorem A and Theorem \ref{thm:q-ana_asymp}, using the semidefinite programming methods described in \cite{CGL}. However, this did not lead to any improvement over the results obtained with bandlimited functions in $\mathcal{A}_1$, even after using polynomials of large degrees, and significantly larger than the degree used in \cite{CGL}. A similar situation occurred in \cite{CQ-H}, where the aforementioned results in \cite{CMS} and \cite{CPL} were further extended to primes represented by quadratic forms. Therein, bandlimited functions also outperform polynomials times gaussians, unless one uses much larger degrees, which might not be feasible. 
	
	Nevertheless, for completeness, we will briefly describe how to construct these functions with semidefinite programming in the present framework, and the results obtained. Henceforth, assume GRH. In \cite{GGOS}, the authors show that, for any fixed, small $\delta>0$, we have
	\begin{equation}\label{eq:F-grh}
	    F(\alpha, \, T) \ge \frac{3}{2} - |\alpha| - o(1),
	\end{equation}
	uniformly for $1\le |\alpha|\le \frac{3}{2}-\delta$, as $T\to\infty.$ This gives a conditional improvement over the asymptotic formula \eqref{eq:F_asymp}, and has been used to refine some estimates under GRH (e.g. in \cite{CGL,GGOS}). Using \eqref{eq:F-grh}, by an argument similar to that of Section \ref{sec:asymptotic-lower} and \cite[Section 2.4]{CCCM}, we find the following: for fixed $b\ge 1$ and $\ell$ sufficiently large,
	\begin{equation*}
	    \frac{1}{\ell}\int_b^{b+\ell} F(\alpha, \, T)\, \d \alpha \ge J_1(g)-\varepsilon +o(1), 
	\end{equation*}
	as $T\to\infty$, and
	\begin{equation*}
	    \frac{1}{\ell}\int_b^{b+\ell} F_\Phi(\alpha, \, Q)\, \d \alpha \ge J_2(g)-\varepsilon +o(1), 
	\end{equation*}
	as $Q\to\infty$, for any $g\in \mathcal{A}$. Here, we denote
\begin{equation}\label{eq:semi-j1}
    J_1(g) := \frac{(1-c_0)g(0) + c_0\left(\rho_1(g) + 2\int_1^{3/2}\left(\frac{3}{2}-\alpha \right)\widehat{g}(\alpha)\,\d \alpha
    \right)}{\max_{0\le \alpha \le 1}\sum_{n=0}^m|\widehat{g}(n-\alpha)|},
\end{equation}
and 
\begin{equation}\label{eq:semi-j2}
    J_2(g) := \frac{
    (1-c_0)g(0) + 
	   \frac{c_0}{2}\left(\widehat{g}(0)+8\int_0^{1/2} \alpha\, \widehat{g}(\alpha)\,\d \alpha +4\int_{1/2}^1 \widehat{g}(\alpha)\,\d \alpha\right) 
    }
    {\max_{0\le \alpha \le 1}\sum_{n=0}^m|\widehat{g}(n-\alpha)|}.
\end{equation}
We may take the parameter $m$ to be any positive integer, and, as in \cite[Section 2.4.2]{CCCM}, the bounds improve as $m\to\infty$. Note that, if $g\in \mathcal{A}_1$, then $J_1$ and $J_2$ simplify to the functionals in (EP2) and (EP4), respectively. Furthermore, note that, since \eqref{eq:F-grh} is an inequality (instead of an asymptotic equality as in \eqref{eq:F_asymp} and Lemma \ref{lem:F_phi}), we used the sign restrictions of $g\in\mathcal{A}$ to obtain the above bound with $\eqref{eq:semi-j1}$.

\smallskip In contrast to \cite{CGL}, the objectives $J_1$ and $J_2$ are not linear (or even smooth). To transform it into a semidefinite program, we approximate our problem by one with a linear objective and additional linear inequality constraints. Let $N$ be a positive integer, and let $\{\alpha_1, \, \alpha_2,\, \ldots, \, \alpha_N\}$ be a partition of the interval $[0,\, 1]$. Then, multiplying $g$ by an appropriate constant, we may replace the denominators of \eqref{eq:semi-j1} and \eqref{eq:semi-j2} by 1, and incorporate the system of inequality constraints 
\[\sum_{n=0}^m\widehat{g}(n-\alpha_j )\le 1,\ \text{ for } \ 1\le j \le N.\] 
When $N$ is sufficiently large, this results in a reasonable approximation in practice.

\smallbreak We now follow the notation and argument in \cite[Section 4]{CGL}, to which we refer for details. By taking dilations, we may relax the condition $\widehat{g}\le 0$ for $|\alpha|\ge 1$, to the same condition over $|\alpha|\ge R$, where $R\ge 1$ is some parameter (after also taking dilations in the definitions of $J_1$ and $J_2$), and we may assume $\widehat{g}(0)=1.$  We see this as a bilevel optimization problem, where the outer problem is a 1-dimensional problem over $R\ge 1$, and the inner problem optimizes over such a function $g(x)=p(x)e^{-\pi x^2}$, where, as before, $p$ is an even polynomial. For a fixed $R$, functions $g$ of this form, that are non-positive in $[R,\, \infty)$, and whose Fourier transform is non-negative, can be written in terms of positive-semidefinite matrices, as follows. \smallskip

Let $d\in \N$. Let $X_2, \, X_3,\, X_4$ be positive-semidefinite matrices of size $(d+1)$, and let 
\[v(u):=\left(L_0^{-1/2}(\pi u), \, \ldots,\, L_d^{-1/2}(\pi u)
\right)\in\R^{d+1},\] 
where $L_k^{-1/2}$ is the Laguerre polynomial of degree $k$ with parameter -1/2. Then, we may write 
\[
g(x) = (R^2-x^2)\, v(x^2)^TX_2\, v(x^2)\, e^{-\pi x^2},\ \text{ and } \ 
\widehat{g}(x)=\left( v(x^2)^TX_3\, v(x^2) + x^2\, v(x^2)^TX_4\, v(x^2)\right)e^{-\pi x^2} .
\]
Note that $g$ is a polynomial of degree $4d+2$, times a Gaussian function. The fact that $\widehat{g}(x)$ is the Fourier transform of $g$ is a linear condition over the entries of $X_2$, $X_3$, and $X_4$ in $\R^{d+1}$, which we also write in terms of the Laguerre basis $v(u)$. \smallskip

This is now a semidefinite program, for which we use the high-precision solver sdpa-gmp \cite{Nak}. In Table 1, we show the maxima of \eqref{eq:semi-j1} and \eqref{eq:semi-j2} for several values of $d$, compared with the bound from the triangle function $\widehat{g}(\alpha)=(1-|\alpha|)_+\in\mathcal{A}_1$. In our computations, we take $m=3$ in the definitions of $J_1$ and $J_2$, and $N=55$ in the partition of $[0, \, 1]$. Further experiments with other values of $m$ and $N$ did not significantly alter the results. For comparison, the authors use $d=40$ in \cite{CGL}.
\begin{table}
\begin{center}
\begin{tabular}{ |c||c|c|c|c|c| }
 \hline
 $d$ & $20$ & $40$ & $60$ & $70$ & $\widehat{g}(\alpha)=(1-|\alpha|)_+$\\
 \hline\hline
 $J_1(g)$ & $0.9211\ldots$ & $0.9236\ldots$ & $0.9245\ldots$ & $0.9248\ldots$ &  $0.9275\ldots$\\
 \hline
 $J_2(g)$ & $0.9748\ldots$ & $0.9774\ldots$ & $0.9784\ldots$ & $0.9788\ldots$ &  $0.9818\ldots$\\
 \hline \hline
\end{tabular}
		\vspace{0.2cm}
\caption{Semidefinite programming bounds for the approximations (with $m=3$ and $N=55$) of $J_1$ and $J_2$, for several parameters $d$, compared with a simple triangle bound. The constructed polynomials have degree $4d+2$. }
\end{center}
\end{table}
	\section*{Acknowledgements}
	The author thanks Emanuel Carneiro, Andr\'es Chirre, and Micah B. Milinovich for helpful discussion, comments, and suggestions. The author also thanks the anonymous referees for helpful comments. The author acknowledges support from CNPq - Brazil and from the STEP Programme of ICTP - Italy.

\end{document}